\documentclass[12pt,a4paper]{article}
\usepackage[dvipdfmx]{graphicx}
\usepackage[abs]{overpic} 
\usepackage{amsfonts}
\usepackage{amsmath}
\usepackage{amsthm}
\usepackage{txfonts}
\usepackage{enumerate}
\usepackage{mathrsfs}
\usepackage{pifont} \usepackage[all]{xy}
\usepackage{graphicx,color}
\usepackage{comment}

\DeclareMathAlphabet{\mathfrak}{U}{euf}{m}{n}
\SetMathAlphabet{\mathfrak}{bold}{U}{euf}{b}{n}
\begin{document}
\theoremstyle{definition}
\newcommand{\Rbb}{\mathbb{R}}
\newcommand{\Nbb}{\mathbb{N}}
\newcommand{\Tbb}{\mathbb{T}}
\newcommand{\Hcal}{\mathcal{H}}
\newcommand{\To}{\Rightarrow}
\renewcommand{\phi}{\varphi}
\renewcommand{\epsilon}{\varepsilon}
\newcommand{\im}{{\rm{im}}}

\newcommand{\too}{\longrightarrow}
\newcommand{\Too}{\Longrightarrow}

\newcommand{\pf}{\textgt{証明}\ }

\newcommand{\qedsq}{$\hspace{\fill}{\square}$\\}

\newcommand{\sub}{\subset}
\newcommand{\cptsub}{\stackrel{\rm{cpt}}{\subset}}
\newcommand{\densesub}{\stackrel{\rm{dense}}{\subset}}
\newcommand{\opensub}{\stackrel{\rm{open}}{\subset}}
\newcommand{\Zbb}{\mathbb{Z}}
\newcommand{\supp}{{{\rm  supp}}}
\newcommand{\del}{\partial}

\newcommand{\Cinf}{{C^\infty}}
\newcommand{\Cinfc}{{C^\infty_c}}
\newcommand{\halfopen}[1]{{[#1)}}

\newtheorem{df}{Definition}[section]
\newtheorem{thm}[df]{Theorem}
\newtheorem{prp}[df]{Proposition}
\newtheorem{lem}[df]{Lemma}
\newtheorem{fact}[df]{Fact}
\newtheorem{cor}[df]{Corollary}
\newtheorem{ex}[df]{Example}
\newtheorem{rem}[df]{Remark}
\newtheorem{claim}[df]{Claim}
\newtheorem{inpA}{Proposition}
\renewcommand{\theinpA}{}
\newtheorem{inpB}{Proposition}
\renewcommand{\theinpB}{}
\newtheorem{inpC}{Proposition}
\renewcommand{\theinpC}{}
\newtheorem{intc}{Corollary}
\renewcommand{\theintc}{}
\newtheorem{intt}{Theorem}
\renewcommand{\theintt}{}
\makeatletter
\renewcommand{\theequation}{%
	\thesection.\arabic{equation}}
\@addtoreset{equation}{section}
\makeatother

%
\begin{center}
	\Large{Magnitude homology of geodesic metric spaces with an upper curvature bound}
\end{center}
\begin{center}{Yasuhiko Asao \footnote{asao(at)ms.u-tokyo.ac.jp}}\end{center}
\begin{center}
  {\it
  Graduate School of Mathematical Sciences, The University of Tokyo \\
}
\end{center}

\begin{abstract}
In this article, we study the magnitude homology of geodesic metric spaces of curvature $\leq \kappa$, especially ${\rm CAT}(\kappa)$ spaces. We will show that the magnitude homology $MH^{l}_{n}(X)$ of such a meric space $X$ vanishes for small $l$ and all $n > 0$. Conseqently, we can compute a total $\mathbb{Z}$-degree magnitude homology for small $l$ for the shperes $\mathbb{S}^{n}$,  the Euclid spaces $\mathbb{E}^{n}$, the hyperbolic spaces $\mathbb{H}^{n}$, and real projective spaces $\mathbb{RP}^{n}$ with the standard metric. We also show that an existence of closed geodesic in a metric space guarantees the non-triviality of magnitude homology.
\end{abstract}
\section{Introduction}

The magnitude is an invariant of enriched categories introduced by Leinster (\cite{L}) as a generalization of the cardinality of sets, the rank of vector spaces, and the Euler number of topological spaces. In particular, for finite metric spaces, magnitude measures the number of efficient points.

The magnitude homology $MH^{l}_{n}$, introduced by Hepworth-Willerton and Leinster-Shulman (\cite{HW}, \cite{LS}),  is a  bi-graded $\mathbb{Z}$-module indexed by non-negative real numbers $l$ and non-negative integers $n$, defined for general metric spaces. As shown by them, for finite metric spaces, the magnitude homology is a categorification of the magnitude. 

Some computations of magnitude homology is studied for graphs by Hepworth-Willerton(\cite{HW}), for convex subsets in $\mathbb{R}^{n}$ by Leinster-Shulman (\cite{LS}), for the geodesic circle by Kaneta-Yoshinaga (\cite{YK}), and for geodesic spheres by Gomi (\cite{G}). 

Since the motivation and the formulation of magnitude homology are algebraic, several authors study it from a view point of algebra and category theory. On the other hand, its geometric meaning is gradually getting clarified in the study of Leinster-Shulman, Kaneta-Yoshinaga, and Gomi (\cite{LS}, \cite{YK}, \cite{G}).

In this article, we study magnitude homology from a view point of metric geometry, and clarify that the curvature of metric spaces effects heavily on the triviality of magnitude homology. We also investigate its connection with closed geodesics because magnitude homology is defined as Hochschild cohomology, which is well known to be deeply related to free loop spaces.

Let $(X, d)$ be a metric space. A quadruple $(x_{0}, x_{1}, x_{2}, x_{3}) \in X^{4}$ is called {\it 4-cut} if (1) $x_{i} \not= x_{i+1}$ for $0 \leq i \leq 2$, (2) $d(x_{i}, x_{i+2}) = \sum_{j = i}^{i+1}d(x_{j}, x_{j+1})$ for $0 \leq i \leq 1$, and (3) $d(x_{0}, x_{3}) < \sum_{j = 0}^{2}d(x_{j}, x_{j+1})$ are satisfied. Let $m_{X}$ be the infimum of the length $\sum_{j = 0}^{2}d(x_{j}, x_{j+1})$ of 4-cuts $(x_{0}, x_{1}, x_{2}, x_{3})$ in $X$. This invariant is first introduced in \cite{YK} and is very important for the study of magnitude homology. We clarify its metric geometrical interpretation by the following theorem stated as Theorem \ref{ineq} in this artcle. The invariant $l_{X}$ measures the infimum length of locally geodesic path which is not a geodesic. See Definition \ref{lx} for the detail. Let  $D_{\kappa}$ be $\pi/\sqrt{\kappa}$ for $\kappa > 0$ or $+\infty$ for $\kappa \leq $0.
\begin{intt}
For a geodesic ${\rm CAT}(\kappa)$ space $(X, d)$, we have 
\[
D_{\kappa} \leq m_{X} \leq l_{X}.
\]
\end{intt}
The following is stated as Corollary \ref{vanish}.
\begin{intc}
Let $(X, d)$ be a geodesic ${\rm CAT}(\kappa)$ space. Then for every $n > 0$ and $0 < l < D_{\kappa}$, the magnitude homology $MH^{l}_{n}(X)$ vanishes.
\end{intc}
Since the magnitude homology is defined as a generalization of Hochschild homology of associative algebras in \cite{LS}, it is interesting to investigate some relation with the existence of closed geodesics. The following is stated as Theorem \ref{clogeo}. See Definition \ref{closed} for the definition of closed geodesics.
\begin{intt}
Let $X$ be a metric space. If there exists a closed geodesic of radius $r$ in $X$, then we have $MH^{\pi r}_{2}(X) \not = 0$.
\end{intt}
We also study a metric space with bouded curvature which is not ${\rm CAT}(\kappa)$. The {\it injectiviti radius} $\iota_{X}$ is the supremum of $r$ such that any pair of points distance $<r$ apart can be joined by the unique geodesic.  The following is stated as Proposition \ref{mi}.
\begin{intt}
Let $(X ,d)$ be a proper geodesic metric space of curvature $\leq \kappa$. If the injectivity radius $\iota_{X}$ is no greater than $D_{\kappa}$, then we have  
\[
 \iota_{X} \leq m_{X}.
\]
\end{intt}
The {\it systole} ${\rm Sys}(X)$ is the infimum of length of closed geodesics in $X$. The following is stated as Corollary \ref{locat}.
\begin{intc}
Let $X$ be a cocompact proper geodesic metric space of curvature $\leq \kappa$ which is not {\rm CAT}$(\kappa)$, or the standard sphere. Then for every $n > 0$ and every $0 < l < \iota_{X} = {\rm Sys}(X)/2$, the magnitude homology $MH^{l}_{n}(X)$ vanishes.
\end{intc}
This article is organized as follows. In section 2, we briefly recall some notation and fundamental facts on metric geometry. In section 3, we briefly recall the definition of magnitude homology and some facts on them studied in \cite{YK}. In section 4, we study the magnitude homology of ${\rm CAT}(\kappa)$ spaces through the invariants $m_{X}$ and $l_{X}$. We also compute the magnitude homology of some simply connected complete Riemannian manifolds as examples. In section 5, we study how closed geodesic in a metric space effects on the magnitude homology. This is motivated by the fact magnitude homology is defined as a generalization of Hochschild homology in \cite{LS}. In section 6, we study the magnitude homology of non-${\rm CAT}(\kappa)$ metric spaces whose curvature is bouded from above through the invariants $\iota_{X}$ and ${\rm Sys}(X)$. We also compute the magnitude homology of real projetive spaces $\mathbb{RP}^{n}$ with standard metric as an example.
\subsection*{Acknowledgement}The author would like to thank his supervisor Prof. Toshitake Kohno for fruitful comments and enormous support. This work is supported by the Program for Leading Graduate Schools, MEXT, Japan. This work is also supported by JSPS KAKENHI Grant Number 17J01956.
\section{Preliminary on metric geometry}
In this section, we briefly recall some notations on metric geometry from \cite{BH}.
\begin{df}
Let $(X, d)$ be a meric space, and $c$ be a positive number. A path $\gamma \colon [0, c] \to X$ is a {\it linearly reparametrized geodesic connecting} $\gamma(0)$ {\it and} $\gamma(1)$ if it satisfies $d(\gamma(t), \gamma(s)) = \lambda |t - s|$ for every $0\leq s \leq t \leq 1$ and for some $\lambda \geq 0$. When $\lambda = 1$ we call such a path {\it geodesic}. 
\end{df}
\begin{df}
A metric space $X$ is {\it geodesic} if for each pair of points $a, b \in X$, there exist a geodesic connecting them. Furthermore, if such a geodesic is unique, then $X$ is called {\it uniquely geodesic}.
\end{df}
We sometimes denote a geodesic from some interval connecting $a$ and $b$ in a metric space $X$, by $[a, b]$.
\begin{df}
Let $a, b, c$ be points in a metric space $X$, and let $[a, b], [b, c], [c, a]$ be geodesics from some intervals connecting each pair of points. The union of these geodesic images $\Delta_{abc} := [a, b]\cup [b, c]\cup [c, a]$ is called a {\it geodesic triangle}. 
\end{df}
Let $S_{\kappa}$ be a simply connected surface of constant sectional curvature $\kappa$ for $\kappa \in \mathbb{R}$. Note that for every geodesic triangle $\Delta_{abc}$ in a metric space $X$, there exist a geodesic triangle $\tilde{\Delta}_{abc}$ in $S_{\kappa}$ whose sides have the length precisely equal to the length of corresponding sides of $\Delta_{abc}$. We also note that if the inequality 
\[
d(a, b) + d(b, c) + d(c, a) < 2D_{\kappa}
\]
 is satisfied for
 \[
 D_{\kappa} := \begin{cases} \pi/\sqrt{\kappa} & \kappa > 0 \\ +\infty & \kappa \leq 0\end{cases},
 \]
 then such a triangle is uniquely determined up to congruence. The following is a fundamental fact on elementary geometry known as Alexandrov's lemma.
 \begin{prp}[\cite{BH} Lemma I.2.16]\label{alex}
 Let $A, B, B', C$ be distinct points in $S_{\kappa}$. If $\kappa > 0$, we assume that $d(A, B) + d(B, C) + d(C, B') + d(B', A) < 2D_{\kappa}$ and $d(B, C) + d(C, B') < D_{\kappa}$. We suppose that $B$ and $B'$ lie on opposite sides of the line through $A$ and $C$. Let $\alpha, \beta, \gamma ({\it respectively} \ \alpha', \beta', \gamma')$ be the angles of a triangle $\Delta ({\it resp.} \ \Delta')$ with vertices $A, B, C ({\it resp.} \ A, B', C)$. We assume that $\gamma + \gamma' \geq \pi$. Let $\overline{\Delta}$ be a triangle on $S_{\kappa}$ with vertices $\overline{A}, \overline{B}, \overline{B'}$ such that $d(\overline{A}, \overline{B}) = d(A, B), d(\overline{A}, \overline{B'}) = d(A, B')$ and $d(\overline{B}, \overline{B'}) = d(B, C) + d(C, B')$. Let $\overline{\alpha}, \overline{\beta}, \overline{\beta'}$ be the angles of $\overline{\Delta}$ at $\overline{A}, \overline{B}, \overline{B'}$. Then we have 
 \[
 \overline{\beta} \geq \beta, \ \overline{\beta'} \geq \beta'.
 \]
 \end{prp}
\begin{df}
A geodesic triangle $\Delta_{abc}$ in a metric space $X$ is $\kappa$-{\it small} if the inequality $d(a, b) + d(b, c) + d(c, a) < 2D_{\kappa}$ is satisfied.
For a $\kappa$-small geodesic triangle $\Delta_{abc}$, the corresponding triangle $\tilde{\Delta}_{abc}$ in $S_{\kappa}$ is called the {\it comparison triangle in} $S_{\kappa}$. 
\end{df}
\begin{df}
Let $\Delta_{abc} = [a, b]\cup [b, c]\cup [c, a]$ be a geodesic triangle in a metric space $(X, d)$, and let $\tilde{\Delta}_{abc} = [\tilde{a}, \tilde{b}]\cup [\tilde{b}, \tilde{c}]\cup [\tilde{c}, \tilde{a}]$ be its comparison trinangle in $S_{\kappa}$. A point $\tilde{s} \in [\tilde{a}, \tilde{b}]$ is {\it the comparison point} of $s \in \Delta_{abc}$ if $s \in [a, b]$ and $d(\tilde{a}, \tilde{s}) = d(a, s)$ is satisfied. We similary define the comparison points for points on $[\tilde{b}, \tilde{c}]$ and $[\tilde{c}, \tilde{a}]$.
\end{df}
The following notion of ${\rm CAT}(\kappa)$ space plays a fundamental role in this paper.
\begin{df}
A metric space $(X, d)$ is ${\rm CAT}(\kappa)$ if for every $\kappa$-small geodesic triangle $\Delta_{abc}$ in $X$ and for every pair of points $s, t \in \Delta_{abc}$, the ${\rm CAT}(\kappa)$ {\it inequality}
\[
d(s, t) \leq d_{S_{\kappa}}(\tilde{s}, \tilde{t})
\] 
holds for the comparison points $\tilde{s}$ and $\tilde{t}$.
\end{df}
\begin{ex}
Apparently, the $2$-sphere $S_{\kappa}$ of constant sectional curvature $\kappa$, the Euclid plane $S_{0}$, and the hyperbolic plane $S_{-\kappa}$ of sectional curvature $-\kappa < 0$ are ${\rm CAT}(\kappa), {\rm CAT}(0)$, and ${\rm CAT}(-\kappa)$ respectively. 
\end{ex}
More generally, we have the following.
\begin{prp}
A complete simply connected Riemannian manifold with sectional curvature $\leq\kappa$ is ${\rm CAT}(\kappa)$.
\end{prp}
Hence for example, we have the following examples of ${\rm CAT}(\kappa)$ spaces. Let $\mathbb{S}^{n}_{\kappa} = \{(x_{0}, \dots, x_{n}) \in \mathbb{R}^{n+1} \mid \sum_{i = 0}^{n}x_{i}^{2} = 1/\kappa\}$ be the $n$-sphere of radius $1/\sqrt{\kappa}$ equipped with the geodesic metric.
\begin{ex}\label{ex1}
Let $n\geq 2$. The $n$-sphere $\mathbb{S}^{n}_{\kappa}$ of radius $1/\sqrt{\kappa}$, the $n$-dimensional Euclid space $\mathbb{E}^{n}$, and the $n$-dimensional hyperbolic space $\mathbb{H}^{n}$ of sectional curvature $-\kappa < 0$ are ${\rm CAT}(\kappa), {\rm CAT}(0)$, and ${\rm CAT}(-\kappa)$ respectively. 
\end{ex}
We denote the circle $ \{(x_{0}, x_{1}) \in \mathbb{R}^{2} \mid x_{0}^{2} + x_{1}^{2} = r^{2}\}$ of radius $r$ equipped with the geodesic metric by $C_{r}$. The following is immediate.
\begin{ex}\label{ex2}
The circle $C_{1/\sqrt{\kappa}}$ is ${\rm CAT}(\kappa)$.
\end{ex}
A connected undirected metric graph with no cycles is called {\it a tree}.
\begin{ex}\label{ex3}
Every tree is ${\rm CAT}(0)$.
\end{ex}
We introduce the following notion of the angle.
\begin{df}
Let $c_{1}$ and $c_{2}$ be geodesics in $X$ with the same start point $C$. We define the {\it angle} between $c_{1}$ and $c_{2}$ at $C$ by 
\[
\angle C := {\rm limsup}_{\epsilon, \epsilon' \to 0}\tilde{\angle} c_{1}(\epsilon)Cc_{2}(\epsilon'),
\]
where $\tilde{\angle} c_{1}(\epsilon)Cc_{2}(\epsilon')$ is the angle of the comparison triangle $\tilde{\Delta}_{c_{1}(\epsilon)Cc_{2}(\epsilon')}$ at $\tilde{C}$ in the Euclid plane $S_{0}$. 
\end{df}
The following are fundamental. See for example \cite{BH} for the proof.
\begin{prp}[\cite{BH} Proposition I.1.14]\label{anglesum}
Let $X$ be a metric space and let $c, c'$ and $c''$ be geodesics in $X$ starting from the same point $p$. Then we have 
\[
\angle(c', c'') \leq \angle(c, c') + \angle(c, c'').
\]
\end{prp}
\begin{prp}[\cite{BH} Proposition II.1.7]\label{angle}
Let $\kappa \in \mathbb{R}$. For a geodesic metric space $(X, d)$, the following are equivalent.
\begin{itemize}
\item[(i)] $(X, d)$ is {\rm CAT}$(\kappa)$. 
\item[(ii)] For every $\kappa$-small geodesic triangle $\Delta_{abc}$, the angles at $a, b$ and $c$ are no greater than the corresponding angles of the comparison triangle $\tilde{\Delta}_{abc}$ in $S_{\kappa}$.
\end{itemize}
\end{prp}
The following proposition is technical but singnificant in this paper.
\begin{prp}[\cite{BH} Lemma II.4.11]\label{smoothin}
Let $\kappa$ be a real number, and $X$ be a metric space of curvature $\leq \kappa$. Let $q \colon [0, 1] \to X$ be a linearly  reparametrized geodesic connecting two distinct points $q(0)$ and $q(1)$, and let $p$ be a point in $X$ which is not on the image of $q$. Assume that for each $s \in [0, 1]$, there is a linearly reparametrized geodesic $c_{s} \colon [0, 1] \to X$ connecting $p$ and $q(s)$, varying continuously with $s$. We further assume that the geodesic triangle $\Delta_{q(0)pq(1)}$ is $\kappa$-small. Then the angles of $\Delta_{q(0)pq(1)}$ at $q(0), p$ and $q(1)$ are no greater than the corresponding angles of any comparison triangle $\tilde{\Delta}_{q(0)pq(1)}$ in $S_{\kappa}$.
\end{prp}
We study not only {\rm CAT}$(\kappa)$ spaces, but also locally {\rm CAT}$(\kappa)$ spaces. We recall some fundamental notions on them.
\begin{df}
A metric space $X$ is of {\it curvature} $\leq \kappa$ if for each point $x \in X$ there exist $r_{x} > 0$ such that the ball $B(x, r_{x})$ with the induced metric is ${\rm CAT}(\kappa)$.
\end{df}
The following well-known fact due to Alexandrov supports the significance for studying metric spaces of curvature $\leq \kappa$.
\begin{prp}[\cite{BH} Theorem I.1A.6]
A smooth Riemannian manifold is of curvature $\leq \kappa$ if and only if its sectional curvature is $\leq \kappa$.
\end{prp}
\begin{ex}\label{ex4}
The standard projective space $\mathbb{RP}^{n}$ for $n \geq 2$ is of curvature $\geq 1$. Furthremore, it is {\bf not} ${\rm CAT}(1)$ since a closed geodesic which lifts to the geodesic semi-circle on $\mathbb{S}^{n}$ does not satisfy the ${\rm CAT}(1)$ angle condition.
\end{ex}
\begin{df}
For a metric space $X$, the {\it injectivity radius} $\iota_{X}$ is the supremum of $r \geq 0$ such that any two point of distance $< r$ is connected by the unique geodesic. The {\it systole} ${\rm Sys}(X)$ is the infimum of the length of closed geodesic in $X$ if there exist some, or 0 otherwise.
\end{df}
The following propositions shows the significance of the notions of the injectivity radius and the systole. See for example \cite{BH} for the proof. A metric space $X$ is called {\it cocompact} if there exist a compact subset $K \subset X$ such that $X = \bigcup_{f \in {\rm Isom}(X)} fK$ holds.
\begin{prp}[\cite{BH} Proposition II.4.16]\label{gromov}
Let $X$ be a cocompact proper geodesic metric space of curvature $\leq \kappa$. Then $X$ fails to be ${\rm CAT}(\kappa)$ if and only if there exists a closed geodesic of length $< 2D_{\kappa}$. Moreover, if there exist such a closed geodesic, then there exist a closed geodesic of length ${\rm Sys}(X) = 2\iota_{X}$.
\end{prp}
\section{Preliminary on magnitude homology}
In this section, we briefly recall some notations on magnitude homology from \cite{YK}. Through this section, $(X, d)$ denotes a metric space unless otherwise mentioned.
\begin{df}
An $(n+1)$-tuple $(x_{0}, \dots, x_{n}) \in X^{n+1}$ is called an $n$-{\it chain} of $X$. An $n$-chain is {\it proper} if $x_{i} \not = x_{i+1}$ for all $0\leq i \leq n$. The {\it length} $|x|$ of $n$-chain $x$ is defined by 
\[
|x| := \sum_{i=0}^{n-1} d(x_{i}, x_{i+1}).
\]
\end{df}
We denote the set of all proper $n$-chains of length $l$ by $P^{l}_{n}(X)$, and $P_{n}(X)$ denotes the union of them running through all $l \geq 0$. Let $MC^{l}_{n}(X)$ be the abelian group freely generated by $P^{l}_{n}(X)$.
\begin{df}
Let $a$ and $b$ be points in $X$. A point $c \in X$ is {\it smooth between} $a$ {\it and} $b$ if the equality $d(a, b) = d(a, c) + d(c, b)$ holds. We denote $a \prec c \prec b$ if $c$ is a smooth point between $a$ and $b$ with $a \not = c$ and $b\not = c$.
\end{df}
\begin{df}
The {\it magnitude chain complex} $(MC^{l}_{*}(X), \partial_{*}:= \sum_{i=1}^{*} (-1)^{i}\partial^{i}_{n})$ is defined by 
\[
\partial^{i}_{n}(x_{0}, \dots, x_{n}) = \begin{cases}(x_{0}, \dots, \hat{x}_{i}, \dots, x_{n}) & ({\rm if\ } x_{i-1} \prec x_{i} \prec x_{i+1}) \\ 0 & ({\rm otherwise}). \end{cases}
\]
The homology group of magnitude chain complex is called the {\it magnitude homology group} of $X$, and denoted by $MH^{l}_{*}(X)$.
\end{df}
\begin{df}
If the point $x_{i}$ of $x = (x_{0}, \dots, x_{n}) \in P_{n}(X)$ is {\bf not} a smooth point between $x_{i-1}$ and $x_{i+1}$, then we call it a {\it singular point} of $x$. We set the endpoints $x_{0}$ and $x_{n}$ singular points. Let $\varphi(x) = (x_{s_{0}} = x_{0}, x_{s_{1}}, \dots, x_{s_{k}}= x_{n}) \in P_{k}(X)$ be the tuple of all singular points of $x$. We call $\varphi(x)$ the {\it frame} of $x$. A chain $x \in P_{n}^{l}(X)$ is {\it geodesically simple} if $|\varphi(x)| = |x|$ holds.
\end{df}
Let $P^{F}_{n}(X)$ be the set of all geodesically simple $n$-chains whose frame is $F \in P^{|F|}_{\leq n}(X) : = \bigcup_{k \leq n}P^{|F|}_{k}(X)$. We denote the abelian group freely generated by $P^{F}_{n}(X)$ by $MC^{F}_{n}(X)$. We set  
\[
MC^{{\rm simp}, l}_{n}(X) := \bigoplus_{F \in P^{l}_{\leq n}(X)} MC^{F}_{n}(X).
\]
Note that both $MC^{F}_{*}(X)$ and $MC^{{\rm simp}, l}_{n}(X)$ are subcomplexes of $MC^{*}_{*}(X)$, and we denote their homology by  $MH^{F}_{*}(X)$ and $MH^{{\rm simp}, l}_{n}(X)$ respectively. 
\begin{df}
A proper $3$-chain $x = (x_{0}, x_{1}, x_{2}, x_{3})$ is a {\it $4$-cut} if $\varphi(x) = (x_{0}, x_{3})$ and $d(x_{0}, x_{3}) < |x|$ holds.
\end{df}
\begin{df}\label{4cut}
We define $m_{X}$ to be the infimum of the length of 4-cuts in $X$.
\end{df}
The following theorem is shown in \cite{YK}.
\begin{thm}[\cite{YK} Theorem 3.12, Theorem 5.11]\label{yk1}

\begin{itemize}
\item[(1)] 
\[
MH^{{\rm simp}, l}_{n}(X) \cong \bigoplus_{F \in P^{l}_{\leq n}} MH^{F}_{n}(X).
\]
\item[(2)] For $n > 0$ and $0 < l < m_{X}$,
\[
MH^{{\rm simp}, l}_{n}(X) \cong MH^{l}_{n}(X).
\]
\item[(3)] For a proper $1$-chain $F = (x_{0}, x_{1})$, the natural map 
\[
MH^{F}_{n}(X) \to MH^{|F|}_{n}(X)
\]
is injective.
\end{itemize}
\end{thm}
\begin{df}
Let $a$ and $b$ be points in $X$. The {\it interval poset} $I(a, b)$ is a poset which consists of smooth points between $a$ and $b$, and the partial order $\leq$ among them is defined by 
\[
x \leq y \Leftrightarrow a \prec x \preceq y .
\]
Note that this definition is equvalent to
\[
x \leq y \Leftrightarrow x \preceq y \prec b.
\]
\end{df}
We recall the definition of the order complex and its reduced chain complex of a poset.
\begin{df}
Let $P$ be a poset. The {\it order complex of } $P$ denoted by $\Delta(P)$ is the abstract simplicial complex whose $n$-simplices are the subsets $\{x_{0}, \dots, x_{n}\}$ of P such that  $x_{0} \prec \cdots \prec x_{n}$. Its {\it reduced chain complex} denoted by $(C_{*}(\Delta(P)), \partial_{*})$ is defined by 
\[
C_{n}(\Delta(P)) = \begin{cases} \bigoplus_{x_{0}\prec \cdots \prec x_{n}}\mathbb{Z}\langle x_{0}, \dots, x_{n} \rangle  &n\geq 0,\\ \mathbb{Z} & n = -1, \\ 0 & n < -1,\end{cases}
\]
and $\partial_{n} = \sum_{i=0}^{n} (-1)^{i}\partial_{n}^{i}$ with
\[
\partial_{n}^{i}(\langle x_{0}, \dots, x_{n} \rangle) = \begin{cases} \langle x_{0}, \dots , \hat{x_{i}}, \dots, x_{n}\rangle & n\geq 0,\\ 0 & n < 0.\end{cases}
\]
\end{df}
The following theorem is also shown in \cite{YK}.
\begin{thm}[\cite{YK} Corollary 4.5]\label{yk2}
For a proper $m$-chain $F = (x_{0}, \dots, x_{m})$, we have 
\[
MH^{F}_{n}(X) \cong H_{n-2m} (C_{*}(\Delta_{0,1})\otimes C_{*}(\Delta_{1, 2}) \otimes \cdots \otimes C_{*}(\Delta_{m-1, m})),
\]
where $C_{*}(\Delta_{i, i+1})$ is the reduced chain complex of ordered complex $\Delta(I(x_{i}, x_{i+1}))$.
\end{thm}
\section{Magnitude homology of ${\rm CAT}(\kappa)$ spaces}
In this section, we study the magnitude homology of ${\rm CAT}(\kappa)$ space $X$. We will show that the magnitude homology $MH^{l}_{n}(X)$ vanishes for $0 < l < D_{\kappa}$ and $n > 0$. Conseqently, we can compute a total $\mathbb{Z}$-degree magnitude homology for some length $l$ for the spaces in Examples \ref{ex1}, \ref{ex2}, and \ref{ex3}. For the purpose, we introduce a quantity $l_{X}$ for a metric space $X$.
\begin{df}
Let $X$ be a metric space. A continuous map $\gamma \colon [0, c] \to X$ is {\it locally geodesic} if for every $t \in [0, c]$, there exists a neighborhood $U$ of $t$ such that $\gamma|_{U}$ is a geodesic.
\end{df}
\begin{df}
Let $(X, d)$ be a metric space, and let $\gamma \colon [0, c] \to X$ be a continuous map. We define {\it the length of} $\gamma$ by
\[
|\gamma| := \sup_{0 = t_{0} \leq t_{1} \leq \cdots \leq t_{n} = c}\sum_{i=0}^{n} d(\gamma(t_{i}), \gamma(t_{i+1})),
\]
where the supremum is taken over all partitions of $[0, c]$.
\end{df}
\begin{df}\label{lx}
For a metric space $X$, we define $l_{X}$ to be the infimum of length of locally geodesic paths which is not geodesics.
\end{df}
\begin{lem}\label{lem1}
For a geodesic ${\rm CAT}(\kappa)$ space $(X, d)$, we have 
\[
D_{\kappa} \leq l_{X}.
\]
\end{lem}
\begin{proof}
Assume $l_{X} < D_{\kappa}$. Then there exists a map $\gamma \colon [0, c] \to X$ which is locally geodesic but not a geodesic, satifying $l_{X} \leq |\gamma| < D_{\kappa}$. Let $a$ be the supremum of the numbers $0< t < c$ such that $\gamma|_{[0, t]}$ is a geodesic. Then $\gamma|_{[0, a]}$ is a geodesic by the continuity of $d$. Note that $a$ is positive since $\gamma$ is locally geodesic. Let $b$ be the supremum of numbers $a < t \leq c$ such that $\gamma|_{[a, t]}$ is a geodesic. Then $\gamma|_{[a, b]}$ is also a geodesic, and $b$ is positive. Let $\delta$ be a geodesic between $\gamma(0)$ and $\gamma(b)$. (See Figure \ref{fig1}.) By the assumption $|\gamma| < D_{\kappa}$, the geodesic triangle $\Delta_{\gamma(0), \gamma(a), \gamma(b)}$ satisfies $d(\gamma(0), \gamma(a)) + d(\gamma(a), \gamma(b)) + d(\gamma(b), \gamma(0)) < 2D_{\kappa}$. Hence $\Delta_{\gamma(0), \gamma(a), \gamma(b)}$ is $\kappa$-small. Note that its angle at $\gamma(a)$ is $\pi$ because $\gamma$ is locally geodesic. Therefore by the {\rm CAT}$(\kappa)$ condition for angles in Proposition \ref{angle}, the angles of the comparison triangle of $\Delta_{\gamma(0), \gamma(a), \gamma(b)}$ at $\tilde{\gamma}(0)$ and $\tilde{\gamma}(b)$ are both 0. By the {\rm CAT}$(\kappa)$ condition for length, there is a 1-1 correspondence between $\gamma|_{[0, b]}$ and $\delta$, which implies $\gamma|_{[0, b]}$ is a geodesic. This contradicts the definition of $a$.
\end{proof}
\begin{figure}
\centering
        \def \svgwidth{60mm}
        \includegraphics{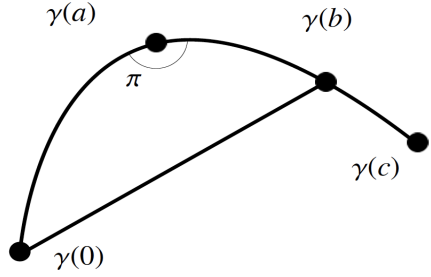}
\caption{Description for the proof of Lemma \ref{lem1}}
\label{fig1}
\end{figure}
Recall that the quantity $m_{X}$ is the infimum of the length of 4-cuts in $X$. (Defitiniton \ref{4cut}.)
\begin{thm}\label{ineq}
For a geodesic ${\rm CAT}(\kappa)$ space $(X, d)$, we have 
\[
D_{\kappa} \leq m_{X} \leq l_{X}.
\]
\end{thm}
\begin{proof}
We first show the former inequality. Assume $m_{X} < D_{\kappa}$. Then there exists a 4-cut $x = (x_{0}, x_{1}, x_{2}, x_{3})$ in $X$ with $m_{X} \leq |x| < D_{\kappa}$. Let $\{ \gamma_{ij} \colon [0, d(x_{i}, x_{j})] \to X \mid 0\leq i < j \leq 3, (i, j) \not= (0, 3) \}$ be a family of geodesics between these points, and $\Delta_{x_{0} x_{1} x_{2}}$ and $\Delta_{x_{1} x_{2} x_{3}}$ be its geodesic triangles. (See Figure \ref{fig1.5}.) Because $|x|$ is smaller than $D_{\kappa}$, the geodesic triangles $\Delta_{x_{0} x_{1} x_{2}}$ and $\Delta_{x_{1} x_{2} x_{3}}$ are both $\kappa$-small. Note that the comparison triangles $\tilde{\Delta}_{x_{0} x_{1} x_{2}}$ and $\tilde{\Delta}_{x_{1} x_{2} x_{3}}$ are both degenerated because they are on some semi-spheres of $S_{\kappa}$. Hecne by the ${\rm CAT}(\kappa)$ inequality, the unions of geodesics $\gamma_{01}\cup \gamma_{12}$ and $\gamma_{12}\cup \gamma_{23}$ coincide with $\gamma_{02}$ and $\gamma_{13}$ respectively. Therefore the map $\gamma_{02}\cup \gamma_{23}$ can be defined and is locally geodesic, which is actually a geodesic because of the inequality $|x| < D_{\kappa} \leq l_{X}$ as shown in Lemma \ref{lem1}. Then we have $d(x_{0}, x_{3}) = d(x_{0}, x_{1}) + d(x_{1}, x_{2}) + d(x_{2}, x_{3})$, which contradicts that the chain $x$ is a 4-cut. Thus we obtain $m_{X} \geq D_{\kappa}$.

Next we show the latter inequality. Assume $l_{X} < m_{X}$. Then there exists a locally geodesic map $\gamma \colon [0, c] \to X$ which is not a geodesic and is satisfying $l_{X} \leq |\gamma| <m_{X}$. Take $a > 0 $ as the supremum of the number $t$ such that $\gamma|_{[0, t]}$ is a geodesic. Let $\epsilon$ be a sufficiently small positive number. Let $b > a$ be the supremum of the number $t$ such that $\gamma|_{[a-\epsilon, t]}$ is a geodesic. (See Figure \ref{fig2}.) If $\big|\gamma|_{[0, b]}\big|$ is smaller than $l_{X}$, then $\gamma|_{[0, b]}$ is a geodesic, which contradicts the assumption $a<c$. Hence we have $\big|\gamma|_{[0, b]}\big| \geq l_{X}$. Note that the proper chain $(\gamma(0), \gamma(a-\epsilon), \gamma(a), \gamma(b))$ has no singular points other than end points, but is not a 4-cut because we have $\big|\gamma|_{[0, b]}\big| \leq |\gamma| < m_{X}$. Similarly, neither is $(\gamma(t_{1}), \gamma(a-\epsilon), \gamma(a), \gamma(t_{2}))$ for $0\leq t_{1} \leq a - \epsilon$ and $a \leq t_{2} \leq b$ . Hence we have  $d(\gamma(t_{2}), \gamma(t_{1})) = (a- \epsilon - t_{1}) + (a - (a - \epsilon)) + (t_{2} - a) = t_{2} - t_{1}$, which implies that $\gamma|_{[0, b]}$ is a geodesic. Thus we obtain a contradiction and conclude $l_{X} \geq m_{X}$.
\end{proof}
\begin{figure}
\centering
        \def \svgwidth{60mm}
        \includegraphics{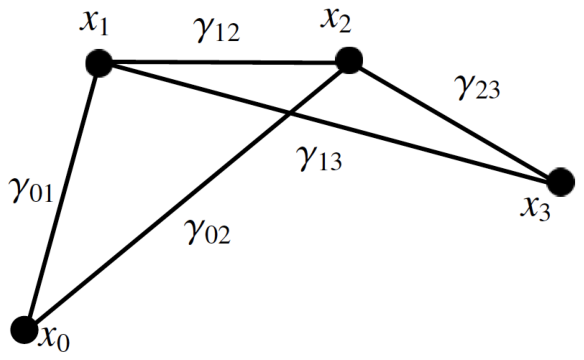}
\caption{Description for the proof of Theorem \ref{ineq}}
\label{fig1.5}
\centering
        \def \svgwidth{60mm}
        \includegraphics{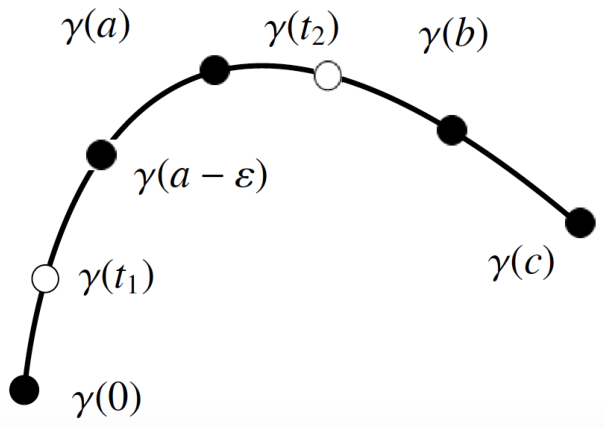}
\caption{Description for the proof of Theorem \ref{ineq}}
\label{fig2}
\end{figure}
\begin{cor}\label{vanish}
Let $(X, d)$ be a geodesic ${\rm CAT}(\kappa)$ space. Then for every $n > 0$ and $0 < l < D_{\kappa}$, the magnitude homology $MH^{l}_{n}(X)$ vanishes.
\end{cor}
\begin{proof}
By Theorem \ref{yk1} $(3)$ and Theorem \ref{ineq}, we have $MH^{l}_{n}(X) \cong \bigoplus _{|F| = l} MH^{F}_{n}(X)$ for $0< l < D_{\kappa}$. We show that every interval poset $I(x, y)$ is totally ordered for $d(x,y) < D_{\kappa}$, which implies that $MH^{F}_{n}(X) = 0$ for $|F| < D_{\kappa}$ by Theorem \ref{yk2}. Let $a$ be a smooth point between $x$ and $y$. Let $\overline{xy}, \overline{xa}$ and $\overline{ay}$ be geodesics connecting each pair of points. Then the geodesic triangle $\Delta_{xay}$ is $\kappa$-small since $d(x, y) < D_{\kappa}$, hence the point $a$ is on $\overline{xy}$. Thus we conclude that $\overline{xy}$ is the unique geodesic connecting $x$ and $y$, and the interval poset $I(x, y)$ is precisely equal to $\overline{xy}$ which is totally ordered.
\end{proof}
\begin{ex}
For the circle $C_{1/\sqrt{\kappa}}$ of radius $1/\sqrt{\kappa}$ and the  $n$-sphere $\mathbb{S}^{n}_{\kappa}$of radius $1/\sqrt{\kappa}$, we have 
\[
MH^{l}_{n}(C_{1/\sqrt{\kappa}}) = MH^{l}_{n}(\mathbb{S}^{n}) = 0,
\]
for $0< l < \pi/\sqrt{\kappa}$ and $n > 0$.
\end{ex}
\begin{ex}
For the $n$-dimensional Euclidean space $\mathbb{E}^{n}$, the $n$-dimensional hyperbolic space $\mathbb{H}^{n}$, and every tree $T$, the magnitude homology $MH^{l}_{n}(\mathbb{E}^{n}), MH^{l}_{n}(\mathbb{H}^{n}),$ and $MH^{l}_{n}(T)$ vanishes for all $l > 0$ and $n > 0$.
\end{ex}
\section{Closed geodesics represent non-trivial Magnitude homology classes}
In this section, we show that an existence of closed geodesic in a metric space $X$ guarantees the non-triviality of the second magnitude homology $MH^{*}_{2}(X)$. As a corollary, we give a criterion of being ${\rm CAT}(\kappa)$ for for a cocompact proper geodesic metric space $X$ of curvature $\leq \kappa$ from a viewpoint of the second magnitude homology. We begin by clarifying the definition of the closed geodesic. 
\begin{df}\label{closed}
Let $X$ be a metric space, and let $C_{r}$ be the circle of radius $r$. An isometry $\rho \colon C_{r} \to X$ is called a {\it closed geodesic of radius} $r$ (or {\it of length} $2\pi r$ ) in $X$. 
\end{df}
\begin{prp}\label{comp}
Let $(X, d)$ be a metric space, and let $\rho \colon C_{r} \to X$ be a closed geodesic. Let $0, 1 \in C_{r}$ be a pair of  antipodal points. Then the interval poset $I(\rho(0), \rho(1))$ has at least two connected components.
\end{prp}
\begin{proof}
Let $U, V$ be semicircles in $C_{r}$ with $ C_{r} = U \cup V$ and $U\cap V = \{ 0, 1\}$. Take points $x \in U-\{0, 1\}$ and $y \in V - \{0, 1\}$, and assume $\rho(x) \leq \rho(y)$ in the poset $I(\rho(0), \rho(1))$. Namely, we have 
\[
d(\rho(0), \rho(y)) = d(\rho(0), \rho(x)) + d(\rho(x), \rho(y)).
\]
Since we have $\rho(0) \prec \rho(x) \prec \rho(1)$ and $\rho(0) \prec \rho(y) \prec \rho(1)$, we obtain
\begin{align*}
d(\rho(x), \rho(1)) & =  d(\rho(0), \rho(1)) - d(\rho(0), \rho(x)) \\
& =  d(\rho(0), \rho(1)) - d(\rho(0), \rho(y)) + d(\rho(x), \rho(y))  \\
& =  d(\rho(y), \rho(1)) + d(\rho(x), \rho(y)) \\
\end{align*}
We also have either 
\[
d(\rho(x), \rho(y)) = d(\rho(x), \rho(1)) + d(\rho(1), \rho(y))
\]
or 
\[
d(\rho(x), \rho(y)) = d(\rho(x), \rho(0)) + d(\rho(0), \rho(y)).
\]
Hence we obtain 
\[
d(\rho(y), \rho(1)) = 0,
\]
or 
\[
d(\rho(x), \rho(1)) - d(\rho(x), \rho(0)) = \pi r,
\]
respectively. Each case implies $y  = 1$ or $x = 0$, which is a contradiction. Hence any pair of points in $\rho(U-\{0, 1\})$ and $\rho(V-\{0, 1\})$ are uncomparable, which implies the order complex $\Delta(I(\rho(0), \rho(1)))$ is not connected.
\end{proof}
\begin{thm}\label{clogeo}
Let $X$ be a metric space. If there exists a closed geodesic of radius $r$ in $X$, then we have $MH^{\pi r}_{2}(X) \not = 0$.
\end{thm}
\begin{proof}
Let $\rho$ be a closed geodesic of radius $r$, and let $\rho(0)$ and $\rho(1)$ be antipodal points of it. Then for the proper 1-chain $F = (\rho(0), \rho(1))$, there exist an injection $MH^{F}_{2}(X) \to MH^{|F|}_{2}(X)$ by Theorem \ref{yk1} $(3)$. Furthermore, we have $MH^{F}_{2}(X) \cong H_{0}(C_{*}(\Delta(I(\rho(0), \rho(1)))))$ by Theorem \ref{yk2}, and Proposition \ref{comp} implies this is non-zero. Hence the statement follows.
\end{proof}
\begin{cor}\label{criterion}
Let $X$ be a cocompact proper geodesic metric space of curvature $\leq \kappa$. Then the following are equivalent.
\begin{itemize}
\item[(i)] $X$ fails to be ${\rm CAT}(\kappa)$.  
\item[(ii)] there exist a closed geodesic of length $< 2D_{\kappa}$. 
\item[(iii)] $MH^{< D_{\kappa}}_{2}(X) \not= 0$.
\end{itemize}
\end{cor}
\begin{proof}
By Proposition \ref{gromov}, (i) implies (ii). By Theorem \ref{clogeo}, (ii) implies (iii). By Corollary \ref{vanish}, (iii) implies (i).
\end{proof}
In particular, the equivalence of (i) and (iii) in Corollary \ref{criterion} gives a criterion of being ${\rm CAT}(\kappa)$ for a cocompact proper geodesic metric space ( especially for compact or homogeneous Riemannian manifold ) whose curvature is bounded from above.
\section{Magnitude homology of non-${\rm CAT}(\kappa)$ metric spaces of curvature $\leq \kappa$}
In this section, we study the magnitude homology of proper geodesic metric spaces by using the injectivity radi and the systoles. As a corollary, we obtain a vanishing theorem for magnitude homology of cocompact proper geodesic metric spaces, and give a partial computation of magnitude homology of the projective spaces $\mathbb{RP}^{n}$ with the standard metric.
\begin{prp}\label{sysl}
For a metric space $X$, we have 
\[
2\iota_{X} \leq {\rm Sys}(X),
\]
and
\[
2l_{X} \leq {\rm Sys}(X).
\]
\end{prp}
\begin{proof}
The former inequality is immediate. For the latter, suppose ${\rm Sys}(X) < 2l_{X}$. Then there exist a closed geodesic $c \colon C_{r} \to X$ of length $2\pi r < 2l_{X} - 2\delta$ for small $\delta > 0$. Then the restriction of $c$ on the interval $[0, \pi r+ \delta]$ is locally geodesic but not a geodesic, with length $\pi r + \delta < l_{X}$. Hence we obtain a contradiction.
\end{proof}
\begin{figure}
\centering
        \def \svgwidth{60mm}
         \includegraphics{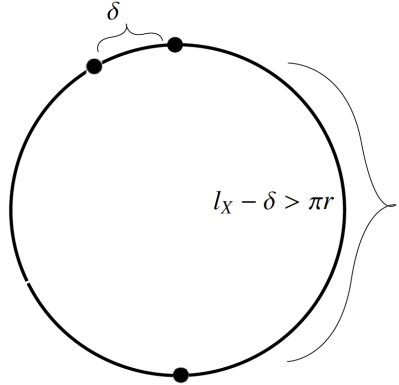}
\caption{Description for the proof of Proposition \ref{sysl}}
\label{fig3}
\end{figure}
\begin{prp}\label{linj}
Let $(X ,d)$ be a proper geodesic metric space of curvature $\leq \kappa$. If the injectivity radius $\iota_{X}$ is no greater than $D_{\kappa}$, then we have 
\[
 \iota_{X} \leq l_{X}.
\]
\end{prp}
For the proof of Proposition \ref{linj}, we use the Arzel\`{a}-Ascoli theorem of the following form. Recall that a sequecne of maps $\{f_{n} \colon Y \to X \}$ between metric spaces is called {\it equicontinuous} if for any positive number $\epsilon$ there exists a positive number $\delta$ such that $d(f_{n}(a), f_{n}(b)) < \epsilon$ holds for every $n$ and $a, b \in Y$ with $d(a, b) < \delta$.
\begin{lem}\label{aa}
Let $X$ be a compact metric space, and $\{\gamma_{n} \colon [0, 1] \to X\}$ be an equicontinuous sequence of maps. Then there exist a subsequence $\{\gamma_{n_{i}}\}$ which uniformly converges to a continuous map $\gamma \colon [0, 1] \to X$.
\end{lem}
\begin{proof}
See for example \cite{BH} Chapter I Lemma 3.10.
\end{proof}
\begin{proof}[Proof of Proposition \ref{linj}.] 
Suppose $l_{X} < \iota_{X}$. Then there exist a locally geodesic path $\gamma \colon [0, L] \to X$ which is not a geodesic, satisfying $l_{X} \leq |\gamma| < \iota_{X}$. Let $0< a <1$ be the supremum of the number $0\leq t < 1$ such that $\gamma|_{[0, t]}$ is a geodesic. Then $\gamma|_{[0, a]}$ is a geodesic by the continuity of $d$. Let $\alpha$ be a positive number such that $\gamma|_{[a, a + \alpha]} $ is a geodesic. We have $\alpha > 0$ by the assumption. We linearly reparametrize $\gamma|_{[0, a]}$ and $\gamma|_{[a, a + \alpha]}$ so that the domain of each map is $[0, 1]$. Let $\delta \colon [0, 1] \to X$ be the linearly reparametrized geodesic connecting $\gamma(0)$ and $\gamma(a + \alpha)$. Let $\gamma_{s} \colon [0, 1] \to X$ be the linearly reparametrized geodesic connecting $\gamma(a)$ and $\delta(s)$ for $0 \leq s \leq 1$. Note that $\gamma|_{[0, a]} = \gamma_{0}$ and $\gamma|_{[a, a + \alpha]} = \gamma_{1}$. We show that the maps $\{\gamma_{s}\}$ varies continuously with $s$. Let $\{s_{n}\}$ be a sequence of points of the interval $[0, 1]$ converging to a point $s' \in [0, 1]$. Let $\Theta$ be the set of subsets $\theta$ in $\mathbb{N}$ such that the subsequence $\{\gamma_{s_{i}} \mid i \in \theta\}$ uniformly converges. We have $\Theta \not= \emptyset$ by Lemma \ref{aa}. Since $d(\gamma(a), \delta(s')) < \iota_{X}$, the convergent curve of $\theta$'s are unique, and it is $\gamma_{s'}$. Hence $\Theta$ is a inductive set, and there exist a maximal element $\overline{\theta}$ in $\Theta$ by Zorn's lemma. If $\mathbb{N} - \overline{\theta}$ is an infinite set, there eixst subset of this which is an element of $\Theta$ by Lemma \ref{aa}. It contradicts to the maximality of $\overline{\theta}$. Therefore the sequence $\{\gamma_{s_{n}}\}$ uniformly coverges to $\gamma_{s'}$, which implies geodesics varies continuously. Then by Proposition \ref{smoothin}, the angles of $\Delta_{\gamma(0)\gamma(a)\gamma(a + \alpha)}$ are no greater than the corresponding angles of comparison triangle in $S_{\kappa}$. Since $\gamma$ is locally geodesic, the angle at $\gamma(a)$ is $\pi$, hence the comparison triangle is degenerated. Thus the angles at $\gamma(0)$ and $\gamma(a + \alpha)$ are both $0$. Let $p$ be a point on $\delta$ apart from $\gamma(a + \alpha)$ by $\alpha$. Then by Proposition \ref{alex} and Proposition \ref{smoothin}, the angle at $\gamma(a + \alpha)$ of the comparison triangle $\tilde{\Delta}_{\gamma(a)p\gamma(a + \alpha)}$ in $S_{\kappa}$ is 0. Hence $\tilde{\Delta}_{\gamma(a)p\gamma(a + \alpha)}$ is degenerated, and we obtain $\gamma(a) = p$. Then we have $\gamma|_{[0, a + \alpha]} = \delta$ by $|\gamma| < \iota_{X}$, which contradicts to the definition of $a$. Therefore we obtain $l_{X} \geq \iota_{X}$.
\end{proof}
\begin{prp}\label{mi}
Let $(X ,d)$ be a proper geodesic metric space of curvature $\leq \kappa$. If the injectivity radius $\iota_{X}$ is no greater than $D_{\kappa}$, then we have  
\[
 \iota_{X} \leq m_{X}.
\]
\end{prp}
\begin{proof}
Suppose $m_{X} < \iota_{X}$. Then there exist a 4-cut $x = (x_{0}, x_{1}, x_{2}, x_{3})$ with $m_{X} \leq |x| < \iota_{X}$. Let $\gamma_{01}, \gamma_{12}$ and $\gamma_{02}$ be the linearly reparametrized geodesics connecting each pair of points $x_{0}, x_{1}$ and $x_{2}$. By the similar argument in the proof of Proposition \ref{linj}, the geodesics connecting $x_{1}$ and points on $\gamma_{02}$ varies continuously. Then by Proposition \ref{smoothin}, the angles of $\Delta_{x_{0}x_{1}x_{2}}$ are no greater than the corresponding angles of comparison triangle in $S_{\kappa}$. Note that the comparison triangle $\tilde{\Delta}_{x_{0}x_{1}x_{2}}$ is degenerated because it is $\kappa$-small and $x_{1}$ is smooth between $x_{0}$ and $x_{2}$, hence $x_{1}$ is on the image of $\gamma_{02}$ by the similar argument in the proof of Proposition \ref{linj}. Similarly, we obtain a geodesic $\gamma_{13}$ connecting $x_{1}$ and $x_{3}$, and going through $x_{2}$. Then geodesics $\gamma_{02}$ and $\gamma_{13}$ coincide between $x_{1}$ and $x_{2}$ by the assumpiton $m_{X} < \iota_{X}$, hence we obtain a locally geodesic path by glueing them. The obtained path turns out to be a geodesic since we have $|x| < \iota_{X} \leq l_{X}$ by Proposition \ref{linj}, which contradicts that $d(x_{0}, x_{3}) < |x|$.
\end{proof}
\begin{cor}\label{locat}
Let $X$ be a cocompact proper geodesic metric space of curvature $\leq \kappa$ which is not {\rm CAT}$(\kappa)$, or the standard sphere. Then for every $n > 0$ and every $0 < l < \iota_{X} = {\rm Sys}(X)/2$, the magnitude homology $MH^{l}_{n}(X)$ vanishes.
\end{cor}
\begin{proof}
For a standard sphere, it follows from Corollary \ref{vanish}. For $X$ being a proper geodesic metric space of curvature $\leq \kappa$ which is not {\rm CAT}$(\kappa)$, we have $\iota_{X} = {\rm Sys}(X)/2 < D_{\kappa}$ by Proposition \ref{gromov}. Furthermore, we have $\iota_{X} \leq m_{X}$ by Proposition \ref{mi}. Hence the statement follows from the similar argument in the proof of Corollary \ref{vanish} and Proposition \ref{mi}.
\end{proof}
\begin{ex}
As mentioned in Example \ref{ex4}, the standard projective space $\mathbb{RP}^{n}$ for $n \geq 2$ is of curvature $\geq 1$ and is not ${\rm CAT}(1)$. Furthermore, we can immediately check that ${\rm Sys}(\mathbb{RP}^{n}) = \pi$. Hence by Corollary \ref{locat}, we obtain
\[
MH^{l}_{\ast}(\mathbb{RP}^{n}) = 0,
\]
for $0 < l < \pi/2$ and $\ast > 0$.
\end{ex}

\end{document}